\documentclass[12pt]{amsart}

\usepackage{latexsym,amsmath,amsfonts,amscd,amsthm,upref,graphics,amssymb}

\newtheorem{thm}{Theorem}[section] 
\newtheorem{lem}[thm]{Lemma}     
\newtheorem{cor}[thm]{Corollary}
\newtheorem{prop}[thm]{Proposition}

\newtheorem{definition}{Definition}

\newtheorem{remark}{Remark}

\newtheorem{ex}{Example}

\begin{document}
\title[A property of ergodic flows] {A property of ergodic flows} 

\author[Maria Joi\c{t}a and Radu-B. Munteanu]{Maria Joi\c{t}a and Radu-B. Munteanu  }
\address{Maria Joi\c{t}a, Department of Mathematics,University of Bucharest,14 Academiei St.,
010014,Bucharest,Romania}
\email{joita@fmi.unibuc.ro }

\address{Radu-B. Munteanu, Department of Mathematics, Univ. of Bucharest,  14 Academiei St.,
010014,Bucharest,Romania}
\email{radu-bogdan.munteanu@g.unibuc.ro}

\date{}




\footnote{This work was supported by a grant of the Romanian
Ministry of Education, CNCS - UEFISCDI, project number
PN-II-RU-PD-2012-3-0533.}

\begin{abstract}
In this paper we introduce a property of ergodic flows, called
Property B. We prove that any ergodic hyperfinite equivalence
relation of type III$_{0}$ whose associated flow satisfies this
property is not of product type. A consequence of this result is
that any properly ergodic flow with Property B is not approximately
transitive. We use Property B to construct a non-AT flow which - up
to conjugacy - is a flow built under a function with the dyadic
odometer as base automorphism.
\end{abstract}
\maketitle
\section{Introduction}\label{intro}

A remarkable result of Krieger \cite{K4} establishes a complete
correspondence between orbit equivalence classes of ergodic
hyperfinite equivalence relations of type III$_{0}$, conjugacy
classes of properly ergodic flows and isomorphism classes of
approximately finite dimensional factors of type III$_{0}$. Product
type equivalence relations are hyperfinite equivalence relations,
which, up to orbit equivalence, are generated by product type
odometers. In order to show that there exist ergodic non-singular
automorphisms not orbit equivalent to any product type odometer,
Krieger \cite{K2} introduced a property of non-singular
automorphisms, called Property A. He proved that any product type
odometer satisfies this property \cite{K3}, and he also constructed
an ergodic non-singular automorphism that does not have this
property, and therefore is not of product type. It was shown in
\cite{M2} that there exist non singular automorphisms which satisfy
Property A but which are not of product type.

To characterize the
ITPFI factors among all approximately finite dimensional factors,
Connes and Woods \cite{CW} introduced a property of ergodic actions,
called approximate transitivity, shortly AT.  They showed that an
approximately finite dimensional factor of type III$_{0}$ is an
ITPFI factor if and only if its flow of weights is AT. Equivalently,
their result says that an ergodic hyperfinite equivalence relation
$\mathcal{R}$ of type III$_{0}$ is of product type if and only if
the associated flow of $\mathcal{R}$ is AT.


In this paper we introduce a property of ergodic flows, called
Property B. We show that any properly ergodic flow with this
property is not AT and we construct a flow which has this property.
The non AT flow corresponding to the non ITPFI factor constructed in
\cite{GH} does not have Property B, and so the property of a flow to
be not AT is not equivalent to Property B.


The paper is organized as follows. In Section 2, we recall some
notations and definitions. In Section 3 we define Property B, we
show that this property is invariant for conjugacy of flows  and we
characterize this property for a flow built under a function. In
Section 4 we prove that a hyperfinite ergodic equivalence relation $\mathcal{R}$ of type III$_{0}$
whose associated flow satisfies Property B is not of product type and we show that a properly ergodic flow which has Property B is not AT.
In Section 5, we show that there exists a flow whith Property B. This flow is built
under a function with the dyadic odometer as base automorphism.

\bigskip

\section{Preliminaries}
Throughout this paper $(X,\mathfrak{B},\mu)$ will be a standard
$\sigma$-finite measure space. A measurable flow on
$(X,\mathfrak{B},\mu)$ is a one parameter group of non-singular
automorphisms $\{F_{t}\}_{t\in\mathbb{R}}$ of $(X,\mathfrak{B},\mu)$
such that the mapping $X\times\mathbb{R}\ni (x,t)\mapsto F_{t}(x)\in X$
is measurable. Two flows $\{F_{t}\}_{t\in\mathbb{R}}$ and
$\{F'_{t}\}_{t\in\mathbb{R}}$ on $(X,\mathfrak{B},\mu)$ and
$(X',\mathfrak{B}',\mu')$ respectively, are conjugate if there
exists an isomorphism
$T:(X,\mathfrak{B},\mu)\rightarrow(X',\mathfrak{B}',\mu')$ such that
for all $t\in \mathbb{R}$ and for $\mu$-almost all $x\in X$,
$F'_{t}(T(x))=T(F_{t}(x))$. We say that $\{F_{t}\}_{t\in\mathbb{R}}$ is
ergodic if any $F_{t}$-invariant measurable set is either null or
conull.


Let $\mathcal{R}$ be an equivalence relation on
$(X,\mathfrak{B},\mu)$. We say that $\mathcal{R}$ is a countable
measured equivalence relation if the equivalence classes
$\mathcal{R}(x)$, $x\in X$ are countable, $\mathcal{R}$ is a
measurable subset of $X\times X$, and the saturation of any set of
measure zero has measure zero. $\mathcal{R}$ is called ergodic if
any invariant set is either null or conull. Recall that if $\nu_{l}$
and $\nu_{r}$ are the left and the right counting measures on
$\mathcal{R}$ we have that ${\nu_{l}}\sim\nu_{r}$ and $\delta(x,y)=
\frac{d\nu_{l}}{d\nu_{r}}(x,y)$ is the Radon-Nikodym cocycle of
$\mu$ with respect to $\mathcal{R}$. We say that the measure $\mu$
is lacunary if there exist $\varepsilon>0$ such that $\delta(x,y)=0$
or $|\delta(x,y)|>\epsilon$, for $(x,y)\in\mathcal{R}$. The full
group $[\mathcal{R}]$ of $\mathcal{R}$ is the group of all
nonsingular automorphisms $V$ of $(X,\mathfrak{B},\mu)$ with
$(x,Vx)\in \mathcal{R}$ for $\mu$-a.e. $x\in X$.

A countable measured equivalence relation $\mathcal{R}$ is called
finite if $\mathcal{R}(x)$ are finite for almost all $x\in X$. We
say that $\mathcal{R}$ is hyperfinite, if there are finite relations
$\mathcal{R}_{n}$ with $\mathcal{R}_{n}\subseteq \mathcal{R}_{n+1}$
and $\cup\mathcal{R}_{n}=\mathcal{R}$, up to a set of measure zero.
We recall that $\mathcal{R}$ is hyperfinite if and only if if there
exists a nonsingular automorphism $T$ on $(X,\mathcal{B},\mu)$ such
that, up to a set of measure zero, $\mathcal{R}$ is equal to the
equivalence relation $\mathcal{R}_{T}=\{(x,T^{n}x),x\in X, n\in
\mathbb{Z}\}$ generated by $T$, that is, $\mathcal{R}(x)=\{T^{n}x,
n\in\mathbb{Z}\}$, for $\mu$-a.e. $x\in X$.


Two countable measured equivalence relations $\mathcal{R}$ and
$\mathcal{R}'$ on $(X,\mathfrak{B},\mu)$ and $(X',\mathfrak{B}'$,
$\mu')$ respectively, are called orbit equivalent if there exists an
isomorphism $S:(X,\mathfrak{B},\mu)\rightarrow
(X',\mathfrak{B}',\mu')$, such that
$S(\mathcal{R}(x))=\mathcal{R}'(Sx)$ for $\mu$-a.e. $x\in X$.

Let $(k_{n})_{n\geq 1}$ be a sequence of positive integers, with
$k_{n}\geq 2$. Consider the infinite product probability space
$(X,\mu)=\prod_{n=1}^{\infty}(X_{n},\mu_{n})$, where
$X_{n}=\{0,1,\ldots k_{n}-1\}$ and $\mu_{n}$ are probability
measures on $X_{n}$ such that $\mu_{n}(x)>0$, for all $x\in X_{n}$.
We recall that the tail equivalence relation $\mathcal{T}$ on
$(X,\mu)$ is defined for $x=(x_{n})_{n\geq 1}$ and $y=(y_{n})_{n\geq
1}$ by
\[(x,y)\in \mathcal{T} \text{ iff there exists }n\geq 1
\text{ such that }x_{i}=y_{i}\text{ for all}i>n.\] It easily can be
observed that, up to a set of measure zero, $\mathcal{T}$ is
generated by the odometer defined on $(X,\mu)$. A countable measured
equivalence relation is said of product type if it is orbit
equivalent to the tail equivalence relation on an infinite product
probability space as above, or equivalently, if it is orbit
equivalent to the equivalence relation generated by a product type
odometer.

An ergodic equivalence relation $\mathcal{R}$ is of type III if
there is no $\sigma$-finite $\mathcal{R}$-invariant measure $\nu$
equivalent to $\mu$. The type III equivalence relations are further
classified in subtypes III$_{\lambda}$, where $0\leq \lambda\leq 1$.
Up to orbit equivalence, for $\lambda\neq 0$, there is only one
hyperfinite equivalence of type III$_{\lambda}$, and this is of
product type.

The orbit equivalence classes of ergodic hyperfinite equivalence
relations of type III$_{0}$ are completely classified by the
conjugacy class of their associated flow.
For more details we refer the reader to \cite{FM} and \cite{KS}.

In order to show that there exists ergodic non-singular
automorphisms not orbit equivalent to any product odometer, Krieger
introduced a property of non-singular automorphisms, called Property
A. This property can be defined for equivalence relations (see
\cite{M1}), as follows. Suppose that $\mathcal{R}$ is a hyperfinite
equivalence relation on $(X,\mathfrak{B},\mu)$. Let $\nu$  be a
$\sigma$-finite measure on $X$, equivalent to $\mu$, and
$\delta_{\nu}$ the corresponding Radon-Nicodym cocycle. For $x\in
A$, define
\[\Lambda_{\nu,A,\mathcal{R}}(x)=\{\log\delta_{\nu}(y,x):
(x,y)\in\mathcal{R}\text{ and }y\in
A\}\]\[=\{\log\frac{d\nu\circ\phi}{d\nu}(x): \phi\in [\mathcal{R}],\
(x,\phi(x))\in\mathcal{R}\text{ and }\phi(x)\in A\}.\]For a
$\sigma$-finite measure $\nu\sim\mu$, $A\in\mathcal{B}$ of positive
measure and $s,\zeta > 0$, set
\begin{eqnarray*}
K_{\nu,\mathcal{R}}(A, s,\zeta)= \{x\in A: (e^{s-\zeta},
e^{s+\zeta})\cap\Lambda_{\nu,A,\mathcal{R}}(x)\neq \emptyset\}\cup \\
\{x\in A: (-e^{s+\zeta}, -e^{s-\zeta})\cap
\Lambda_{\nu,A,\mathcal{R}}(x)\neq \emptyset\}.
\end{eqnarray*}
\begin{definition}
Let $\mathcal{R}$ be a hyperfinite equivalence relation on
$(X,\mathfrak{B},\mu)$. Then $\mathcal{R}$ has Property A if there
exists a $\sigma$-measure $\nu\sim\mu$ and $\eta,\zeta >0$ such
that: every set $A\in \mathfrak{B}$ of positive measure contains a
set $B\in\mathfrak{B}$ of positive measure such that
\[\limsup_{s\rightarrow\infty}K_{\nu,\mathcal{R}}(B,s,\zeta)>\eta\cdot\nu(B).\]
\end{definition}
If $\mathcal{R}$ is a hyperfinite equivalence relation and $T$ is  a
non-singular automorphism such that $\mathcal{R}=\mathcal{R}_{T}$,
up to a null set, it easily can be observed that $\mathcal{R}$ has
Property A if and only if $T$ has Property A (see \cite{M2}). We
mention the following result (see  \cite{K3} and \cite{M2}) that
will be used in this paper.
\begin{prop}\label{alfa}
Assume that $\mathcal{R}$ has Property A. Then there exist
$\eta,\delta>0$ such that for all $\lambda\sim\mu$ and all
$\epsilon>0$, every measurable set $A$ of positive measure contains
a measurable set $B$ of positive measure with
\[\limsup_{s\rightarrow\infty}K_{\lambda,\mathcal{R}}(B,s,\delta+\epsilon)>e^{-\epsilon}\eta\cdot\lambda(B).\]
\end{prop}

\noindent We recall that Krieger's result from \cite{K3}, can be
reformulated in the following way \cite{M2}:
\begin{thm}\label{kra}
Any ergodic equivalence relation of product type and of type III has
Property A.
\end{thm}

\section{Property B}
In this section we define a property of measurable flows that we
call Property B, we show that this is an invariant for conjugacy of
flows, and we characterize this property for a flow built under a
function.

Let $\{F_{t}\}_{t\in\mathbb{R}}$ be a flow of automorphisms of
$(X,\mathfrak{B},\mu)$. For $A\in\mathfrak{B}$ of positive measure
and $s,\delta>0$ we define
\begin{align*}
\Lambda_{F,\delta,s}(A)=\{x\in A, \ \exists t\in
(e^{s-\delta},e^{s+\delta})\cup(-e^{s+\delta},-e^{s-\delta}),
F_{t}(x)\in A\}.
\end{align*}
\begin{definition}
We say that $\{F_{t}\}_{t\in\mathbb{R}}$ has Property B if there exists
a measurable set $A\subseteq X$ of positive measure such that for
all $\delta>0$
\begin{equation}\label{propR}
\limsup_{s\rightarrow\infty}\mu(\Lambda_{F,\delta,s}(A))=0.
\end{equation}
\end{definition}
\begin{prop}\label{2x}
Let $\{F_{t}\}_{t\in\mathbb{R}}$ be a flow on $(X,\mathfrak{B},\mu)$
satisfying Property B, and $\mu'$ a $\sigma$-finite measure
equivalent to $\mu$. Then, the flow $\{F_{t}\}_{t\in\mathbb{R}}$ on
$(X,\mathfrak{B},\mu')$ has Property B.
\end{prop}
\begin{proof}Let $\mu'\sim\mu$ be a $\sigma$-finite measure equivalent to $\mu$ and denote by $f$ the Radon-Nikodym
derivative of $\mu'$ with respect to $\mu$. Thus, $\mu'(A)=\int_{A}f
d\mu$ whenever $A\in\mathfrak{B}$.

Let  $A$ a be measurable set satisfying (\ref{propR}). There exists
a positive integer $k$ such that $\mu(A\cap\{x\in X, \ f(x)<k\})>0$.
Hence, $B=A\cap\{x\in X, \  f(x)<k\}$ is a subset of $A$ of positive
measure, and then, for every $\delta>0$
\[\limsup_{s\rightarrow\infty}\mu(\Lambda_{F,\delta,s}(B))=0.\]
Since $\mu(B)>0$ and $\mu\sim\mu'$ it results that $\mu'(B)>0$.
Notice that
\[\mu'(\Lambda_{F,\delta,s}(B))=\underset{\Lambda_{F,\delta,s}(B)}{\int} f
d\mu<k\underset{\Lambda_{F,\delta,s}(B)}{\int
}d\nu=k\cdot\mu(\Lambda_{F,\delta,s}(B)).\]Consequently,
\[\limsup_{s\rightarrow\infty}\mu'(\Lambda_{F,\delta,s}(B))=0.\]
and therefore, the flow $\{F_{t}\}_{t\in\mathbb{R}}$ on
$(X,\mathfrak{B},\mu')$ has Property B.
\end{proof}

\begin{prop}\label{1x}
\noindent Let $T:(X',\mathfrak{B}',\mu')\rightarrow
(X,\mathfrak{B},\mu)$ be an isomorphism and assume that
$\mu'=\mu\circ T^{-1}$. If $\{F_{t}\}_{t\in\mathbb{R}}$ is a flow on
$(X,\mathfrak{B},\mu)$ that satisfies Property B and
$\{F'_{t}\}_{t\in\mathbb{R}}$ is a flow on $(X',\mathfrak{B}',\mu')$
such that $T( F'_{t}(x))=F_{t}(Tx)$, for all $t\in\mathbb{R}$ and for
$\mu'$-almost all $x\in X'$, then $F'_{t}$ has Property B.
\end{prop}
\begin{proof}
Assume that there exists a measurable subset $A$ of $X$ of positive
measure which satisfies (\ref{propR}). Let $\delta,s>0$. Up to sets
of measure zero, the following equalities hold:
\begin{align*}
\Lambda&_{F',\delta,s}(T^{-1}(A))\\
=&\{x\in T^{-1}(A), \ \exists \
t\in(e^{s-\delta},e^{s+\delta})\cup (-e^{s+\delta},-e^{s-\delta}), F'_{t}(x)\in T^{-1}(A)\}\\
=&\{x\in X', \ Tx\in A, \ \exists \
t\in(e^{s-\delta},e^{s+\delta})\cup (-e^{s+\delta},-e^{s-\delta}), T(F'_{t}(x))\in A\}\\
=&\{x\in X',\ Tx\in A, \ \exists \
t\in(e^{s-\delta},e^{s+\delta})\cup (-e^{s+\delta},-e^{s-\delta}), F_{t}(Tx)\in A\}\\
=&T^{-1}\left(\{y\in A, \ \exists \
t\in(e^{s-\delta},e^{s+\delta})\cup (-e^{s+\delta},-e^{s-\delta}), F_{t}(y)\in A\}\right.\\
=&T^{-1}(\Lambda_{F,\delta,s}(A)).
\end{align*}
Hence,
\begin{equation*}
\mu'(\Lambda_{F',\delta,s}(T^{-1}(A)))=\mu'\circ
T^{-1}(\Lambda_{F,s,\delta}(A))=\mu(\Lambda_{F,\delta,s}(A)).
\end{equation*}
It then follows that for every $\delta>0$, we have
\begin{equation*}
\limsup_{s\rightarrow\infty}\mu'(\Lambda_{F',\delta,s}(T^{-1}(A)))=\limsup_{s\rightarrow\infty}\mu(\Lambda_{F,\delta,s}(A))=0,
\end{equation*}
and therefore, the flow $\{F'_{t}\}_{t\in\mathbb{R}}$ on
$(X',\mathfrak{B}',\mu')$ has Property B.
\end{proof}
\noindent We can prove now the following result:
\begin{prop}
Property B is an invariant for conjugacy of flows.
\end{prop}
\begin{proof}
Let $(X,\mathfrak{B},\mu)$, $(X',\mathfrak{B}',\mu')$ be two
$\sigma$-finite measure spaces and assume that
$\{F_{t}\}_{t\in\mathbb{R}}$ is a flow on $(X,\mathfrak{B},\mu)$ which
satisfies Property B. Let $\{F'_{t}\}_{t\in\mathbb{R}}$ be a flow on
$(X',\mathfrak{B}',\mu')$ which is conjugate to
$\{F_{t}\}_{t\in\mathbb{R}}$. Hence, there exists an isomorphism
$T:(X',\mathfrak{B}',\mu')\rightarrow (X,\mathfrak{B},\mu)$ such
that $F_{t}(Tx)=T(F'_{t}(x))$ for $\mu'$-almost all $x\in X'$ and
for all $t\in\mathbb{R}$. As $T$ is an isomorphism,
$\mu'\sim\mu\circ T^{-1}$. Let $\mu''$ be the measure on $X'$ given
by $\mu''=\mu\circ T^{-1}$. Thus $F_{t}(Tx)=T(F'_{t}(x))$ for
$\mu''$ almost all $x\in X'$ and for all $t\in\mathbb{R}$. By
Proposition \ref{1x}, we have that $\{F'_{t}\}_{t\in\mathbb{R}}$ on
$(X',\mathfrak{B}',\mu'')$ has Property B. As $\mu''$ and $\mu'$ are
equivalent measures, Proposition \ref{2x} implies that
$\{F'_{t}\}_{t\in\mathbb{R}}$ has Property B.
\end{proof}

Let $T$  be an automorphism of $(X_{0},\mathfrak{B}_{0},\mu_{0})$
and $\xi:X_{0}\rightarrow \mathbb{R}$ be a positive measurable
function. Consider $Y=\{(x,t)\in X_{0}\times\mathbb{R}, 0\leq t<
\xi(x)\}$ and let $\nu$ be the measure on $Y$ that is the
restriction of the product measure $\mu_{0}\times\lambda$, where
$\lambda$ is the usual Lebesgue measure on $\mathbb{R}$. Let
$\{F_{t}\}_{t\in\mathbb{R}}$ be the flow built under the function
$\xi$ with  base automorphism $T$; it is defined on $(Y,\nu)$, and
for $t>0$ is given by
\begin{equation*}
F_{t}(x,s)=\left\{
\begin{array}{l}
(x,t+s) \text{ if }0 \leq t+s< \xi(x)\\[0.2cm]
(T(x),t+s-\xi(x)) \text{ if } \xi(x)\leq t+s< \xi(T(x))+\xi(x)\\[0.2cm]
\cdots.
\end{array}
\right.
\end{equation*}

\noindent For a measurable set $A\subseteq X_{0}$ we define
\begin{equation*}
\Delta_{F,\delta,s}(A)=\{ x\in A, \exists
t\in(e^{s-\delta},e^{s+\delta})\cup(-e^{s+\delta},-e^{s-\delta}),
F_{t}(x,0)\in A\times\{0\}\}.
\end{equation*}
With this notation we have the following result:
\begin{prop}\label{PrB}
The flow $\{F_{t}\}_{t\in\mathbb{R}}$ has Property B, if and only if
there exists a measurable set $A_{0}\subseteq X_{0}$ of positive
measure such that, for all $\delta>0$,
\begin{equation}\label{x}
\limsup_{s\rightarrow\infty}\mu_{0}(\Delta_{F,\delta,s}(A_{0}))=0.
\end{equation}
\end{prop}
\begin{proof}
Assume that $\{F_{t}\}_{t\in\mathbb{R}}$ has Property B. Then, there
exists a measurable set $A\subseteq Y$ such that, for every
$\delta>0$,
\begin{eqnarray*}
\limsup_{s\rightarrow\infty}\nu(\Lambda_{F,2\delta,s}(A))=0.
\end{eqnarray*}
Since $A$ has positive measure, there exists a measurable set
$A_{0}\subseteq X_{0}$ of positive measure, an integer $m\geq 1$,
and a positive real $\alpha$ such that for all $x\in A_{0}$,
$\lambda(A_{x}\cap [m,m+1])>\alpha$, where $A_{x}=\{y\in \mathbb{R}:
(x,y)\in A\}$. Let $K=A_{0}\times [m, m+1]\cap A$. Clearly,
$K\subseteq A$, and then, for all $\delta>0$, we have
\begin{eqnarray}\label{3x}
\limsup_{s\rightarrow\infty}\nu(\Lambda_{F,2\delta,s}(K))=0.
\end{eqnarray}
Let $\delta>0$. For any $x\in \Delta_{F,\delta,s}(A_{0})$, there
exists $y\in A_{0}$ and $t\in (e^{s-\delta},e^{s+\delta})\cup
(-e^{s+\delta},-e^{s-\delta})$ such that $F_{t}(x,0)=(y,0)$. For
$(x,a),(y,b)\in K$, we have that
\[F_{t-a+b}(x,a)=F_{t+b}(x,0)=F_{b}(F_{t}(x,0))=F_{b}(y,0)=(y,b).\]
It is straightforward to check that for $s$ large enough, $t-a+b \in
(e^{s-2\delta},e^{s+2\delta})\cup (-e^{s+2\delta},-e^{s-2\delta})$
whenever $t\in (e^{s-\delta},e^{s+\delta})\cup
(-e^{s+\delta},-e^{s-\delta})$. Consequently,
\[\Delta_{F,\delta,s}(A_{0})\times [m,m+1]\cap K\subseteq \Lambda_{F,2\delta,s}(K),\]
whence
\[\alpha\cdot\mu_{0}(\Delta_{F,\delta,s}(A_{0}))\leq \nu(\Lambda_{F,2\delta,s}(K)),\]
and then (\ref{x}) follows from (\ref{3x}).

Conversely, consider $A_{0}\subseteq X_{0}$ satisfying (\ref{x}).
Let $A=A_{0}\times [0,1]\cap Y$. Proceeding in the same manner as
above, for all $\delta>0$, and for $s$ large enough, we have
\[\Lambda_{F,\delta,s}(A)\subseteq \Delta_{F,2\delta,s}(A_{0})\times [0,1]\cap A.\]
Then, by (\ref{x}), we obtain that $\{F_{t}\}_{t\in\mathbb{R}}$
satisfies Property B.
\end{proof}

\section{Property B implies not AT}

In this section we show that if $\mathcal{R}$ is an ergodic
hyperfinite equivalence relation of type III$_{0}$ whose associated
flow has Property B, then $\mathcal{R}$ does not satisfy Krieger's
Property A and therefore is not of product type. A consequence of
this result is that any properly ergodic flow with Property B is not
approximately transitive. Remark that if $\mathcal{R}$ is of type
$III_{\lambda}$, $\lambda\neq 0$, then the associated flow of
$\mathcal{R}$ does not have Property B.

Consider an ergodic hyperfinite equivalence relation $\mathcal{R}$
of type III$_{0}$ on $(X,\mathfrak{B},\mu)$ and let $\delta$ be the
Radon-Nicodym cocycle of $\mu$ with respect to $\mathcal{R}$.
Replacing eventually $\mu$ with an equivalent measure we can assume
that $\mu$ is a lacunary measure (see for example \cite{KW},
Proposition 2.3). Define
$$\xi(x)=\min\{\log\delta(x',x);\ (x',x)\in\mathcal{R},\
\log\delta(x',x)>0\}$$ and consider $\mathcal{S}$ the equivalence
relation on $X$ given by
\[(x,y)\in\mathcal{S}\text{ if and only if }
(x,y)\in\mathcal{R} \text{ and }\delta(x,y)=1.\] Let
$\mathfrak{B}(\mathcal{S})$ the $\sigma$-algebra of sets in
$\mathfrak{B}$ that are $\mathcal{S}$-invariant. Let $X_{0}$ be the
quotient space $X/\mathfrak{B}(\mathcal{S})$, that is the space of
ergodic components of $\mathcal{S}$. We denote the quotient map from
$X$ onto $X_{0}$ by $\pi$, where $\pi(x)$ is the element of $X_{0}$
containing $x$. On $X_{0}$, consider the measure
$\mu_{0}=\mu\circ\pi^{-1}$. Note that $\xi(x)$ is
$\mathfrak{B}(\mathcal{S})$-measurable and therefore, $\xi$ can be
regarded as a function on $X_{0}$. We have an ergodic automorphism
$T$ on $X_{0}$ defined $T(\pi(x))=\pi(x')$ where
$(x,x')\in\mathcal{R}$ and $\log \delta(x',x)=\xi(\pi(x))$. 
Then, the associated flow $\{F_{t}\}_{t\in\mathbb{R}}$ of
$\mathcal{R}$ can be realized as the flow built under the ceiling
function $\xi$ with base automorphism $T$ (see for example
\cite{HO} or \cite{KW}).

\begin{lem}
\label{fl4}
Let $(x,x')\in\mathcal{R}$, $z=\pi(x)$ and $z'=\pi(x')$. Then
$F_{\log\delta(x',x)}(z,0)$ $=(z',0)$.
\end{lem}
\begin{proof}
Notice that it is enough to prove the lemma for $\log\delta(x',x)$
positive. Since $\mu$ is a lacunary measure, there are only finitely
many values, say $n$, of $\log\delta(z,x)$ between $0$ and
$\log\delta(x',x)$. Hence there exists $x_{1}, x_{2},\ldots, x_{n}$
in the orbit $\mathcal{R}(x)$ of $x$ such that
$0<\log\delta(x_{1},x)<\cdots<\log\delta(x_{n},x)<\log\delta(x',x)$.
Then
$$\log\delta(x',x)=\log\delta(x_{1},x)+\log\delta(x_{2},x_{1})+\cdots+\log\delta(x_{n},x_{n-1})+\log\delta(x',x_{n}).$$
If $z_{i}=\pi(x_{i})$, then $z_{i}=T^{i}(z)$ and
$F_{\xi(T^{i-1}(z))}(T^{i-1}(z),0)=(T^{i}(z),0)$, for $1\leq i\leq
n$. Notice that
$\log\delta(x',x)=\xi(z)+\xi(T(z))+\cdots+\xi(T^{n}(z))$.
Therefore\begin{align*}
F_{\log\delta(x',x)}(z,0)&=F_{\xi(z)+\xi(T(z))+\cdots+\xi(T^{n}(z))}(z,0)\\
&=F_{\xi(T(z))+\cdots+\xi(T^{n}(z))}(T(z),0)=\cdots=(z',0).
\end{align*}
\end{proof}

\begin{thm}
With the above notation, if the associated flow
$\{F_{t}\}_{t\in\mathbb{R}}$ of $\mathcal{R}$ has Property B, then
$\mathcal{R}$ does not have Property A.
\end{thm}

\begin{proof}
By Proposition \ref{PrB}, we can find a measurable set
$A_{0}\subseteq X_{0}$ of positive measure such that, for all
$\delta>0$,
\[\limsup_{s\rightarrow\infty}\mu_{0}(\Delta_{F,\delta,s}(A_{0}))=0.\]
Let $C=\pi^{-1}(A_{0})\subseteq X$ and $\delta>0$. Consider $s>0$
and $x\in K_{\mu,\mathcal{R}}(C,s,\delta)$. Thus, $x\in C$ and there
exists $y\in C$ such that $(x,y)\in\mathcal{R}$ and $\log
\delta(y,x)\in (e^{s-\delta}, e^{s+\delta})\cup
(-e^{s+\delta},-e^{s-\delta})$. From Lemma \ref{fl4}, we have that
\[F_{\log \delta(y,x)}(\pi(x),0)=(\pi(y),0).\]
Hence, $\pi(x)\in\Delta_{F,\delta,s}(A_{0})$ and then, $x\in
\pi^{-1}\left(\Delta_{F,\delta,s}(A_{0})\right)$.
Therefore,
\[K_{\mu,\mathcal{R}}(C,s,\delta)\subseteq
\pi^{-1}\left(\Delta_{F,\delta,s}(A_{0})\right),\] and consequently,
\[\mu(K_{\mu,\mathcal{R}}(C,s,\delta))\leq\mu\circ\pi^{-1}(\Delta_{F,\delta,s}(A_{0}))=\mu_{0}(\Delta_{F,s,\delta}(A_{0})).\]
This clearly implies that
\[\limsup_{s\rightarrow\infty}\mu(K_{\mu,\mathcal{R}}(C,s,\delta))=0\]
and then, by Proposition \ref{alfa}, $\mathcal{R}$ does not have
Property A.
\end{proof}
\begin{remark}
Since any product type equivalent relation of type III satisfies
Property A, it follows that an equivalent relation $\mathcal{R}$
whose associated flow has Property B is not of product type.
\end{remark}
Recall that any properly ergodic flow is the associated flow of
certain ergodic hyperfinite equivalence relation of type III$_{0}$
and a hyperfinite ergodic equivalence relation is of product type,
if and only if the associated flow is approximately transitive. We
have then the following result:
\begin{cor}\label{c1}
Let $\{F_{t}\}_{t\in\mathbb{R}}$ be a properly ergodic flow on $(X,\mathfrak{B},\mu)$ which satisfies Property B. Then
$\{F_{t}\}_{t\in\mathbb{R}}$ is not approximately transitive.
\end{cor}
\begin{remark}
There exists ergodic flows which are not AT and do not satisfy
Property B, as the following example shows.
\end{remark}
\begin{ex}
In \cite{GH}, Giordano and Handelman constructed a factor $N$ whose
flow of weights is not AT. We recall that the flow of weights of $N$
can be realized as the flow built under a constant function and
which has a base automorphism that can be identified with the
Poisson boundary of the matrix valued random walk corresponding to
the dimension space given by the sequence of matrices
\[\left[\begin{array}{ccc}
x^{5^{n}} & 1\\
1 & x^{5^{n}}
\end{array}\right], n\geq 1.\]
Since, up to isomorphism, $N$ is the von Neumann algebra associated
to an ergodic hyperfinite equivalence relation $\mathcal{R}$, the
flow of weights of $N$ is, up to conjugacy, the associated flow of
$\mathcal{R}$. According to \cite{M2}, the equivalence relation
$\mathcal{R}$ has Property A, and then, from Theorem 4.2 we
conclude that the associated flow of $\mathcal{R}$ does not satisfy
Property B.
\end{ex}




\noindent The following result gives a sufficient condition for a nonsingular
automorphism to be not AT.
\begin{cor}
Let $T$ be a nonsingular automorphism of $(X,\mathfrak{B},\mu)$.
Assume that there exists $A\subset X$ of positive measure such that
\[\limsup_{s\rightarrow\infty}\{x\in A: \exists n\in (e^{s-\delta},e^{s+\delta})\cup(-e^{s+\delta},-e^{s-\delta}), T^{n}x\in A\}.\] Then $T$ is not AT.
\end{cor}
\begin{proof}
Proposition \ref{PrB} implies that $\{F_{t}\}_{t\in\mathbb{R}}$, the
flow built under the constant function $f=1$ with base automorphism
$T$ has Property B and then by Corollary \ref{c1} it follows that $\{F_{t}\}_{t\in\mathbb{R}}$  is not AT. From Lemma 2.5 of \cite{CW}, we conclude that $T$ is not AT.
\end{proof}
\section{An ergodic flow which satisfies Property B}
In this section we construct a properly ergodic flow which satisfies
Property B and therefore is not AT. The flow that we construct is a
flow built under a function with a product odometer (conjugate to
the dyadic odometer) as base automorphism.

Let $(z_{n})_{n\geq 1}$ be the sequence of integers given by
$z_{n}=2^{n}-1$, for $n\geq 1$. Consider the product space
$X=\prod_{n\geq 1}\{0,1,\ldots z_{n}\}$ endowed with the usual
product $\sigma-$algebra and the product measure $\mu=\otimes_{n\geq
1}\mu_{n}$, where $\mu_{n}$ are the probability measures on
$\{0,1,\ldots z_{n}\}$ given by $\mu_{n}(i)=\frac{1}{2^{n}}$, for
$i=0,1,\ldots,z_{n}$ and $n\geq 1$. Let $T:X\rightarrow X$ be the
product odometer defined on $X$. We recall that $T$ is the
nonsingular automorphism defined for almost every $x\in X$ by
\begin{equation}\label{odo}
(Tx)_{n}=\left\{\begin{array}{c}
                              0 \ \ \ \ \ \ \ \text { if }n<N(x), \\
                              x_{n}+1\text{ if }n=N(x), \\
                              \ x_{n} \ \ \ \ \ \text{ if }n> N(x),
\end{array}
\right.
\end{equation}where $N(x)=\min\{n\geq 1: x_{n}< z_{n} \}$. Notice that $T$ is measure conjugate to the dyadic odometer.

Let $(K_{n})_{n\geq 4}$ be the sequence given by $$K_{n}=1!2!\cdots
n!, \text{ for }n\geq 4,$$ and let $f:X\rightarrow\mathbb{R}$ be the
function defined for almost every $x\in X$ by setting
\begin{equation}\label{eqq0}
f(x)=K_{2^{N+1}+x_{N+1}}
\end{equation}
where $x=(x_{n})_{n\geq 1}$ and $N=N(x)$.

\begin{prop}\label{prop}
Let $n\geq 4$ be a positive integer, $m=[\log_{2}{n}]$ and
$l=n-2^{m}$. For almost every $x\in X$, we have:
\begin{itemize}
\item[(i)] If there exists an integer $k\geq 1$ such that
\begin{equation}\label{eqq1}
K_{n}\leq \sum_{i=0}^{k-1}f(T^{i}x)<K_{n+1},\end{equation} then
$x_{m+1}=l$.
\item[(ii)] If there exists an integer $k\geq 1$ such that
\begin{equation}\label{eqq2}
K_{n}\leq \sum_{i=1}^{k}f(T^{-i}x)<K_{n+1},\end{equation} then
$x_{m+1}=l$.
\end{itemize}
\end{prop}
\begin{proof}
(i) Let $x\in X$ such that $K_{n}\leq
\sum_{i=0}^{k-1}f(T^{i}x)<K_{n+1}$, for some integer $k\geq 1$. Let
$$p=\max\{N(T^{i}x); \  0\leq i\leq k-1\}.$$ Hence, there exists
$j$, $0\leq j <k$ such that $N(T^{j}x)=p$. By (\ref{eqq0}) we have
that $f(T^{j}x)=K_{2^{p+1}+x_{p+1}}$. From (\ref{odo}) we deduce
that $(T^{i}x)_{n}=x_{n}$ for $n>p$ and $1<i\leq k$. Also,
(\ref{odo}) implies that $2\cdot 2^{2}\cdots 2^{p}> k$. Hence,
$$K_{2^{p+1}+x_{p+1}}\leq\sum_{i=0}^{k-1}f(T^{i}x)<2\cdot 2^{2}\cdots 2^{p}\cdot K_{2^{p+1}+x_{p+1}}<K_{2^{p+1}+x_{p+1}+1}.$$
We claim that $n=2^{p+1}+x_{p+1}$. Indeed, if $n<2^{p+1}+x_{n+1}$ we
have $K_{n+1}\leq K_{2^{p+1}+x_{n+1}}\leq\sum_{i=0}^{k-1}f(T^{i}x)$,
which contradicts (\ref{eqq1}). If $n>2^{p+1}+x_{p+1}$, then
$K_{n}\geq K_{2^{p+1}+x_{p+1}+1}>\sum_{i=0}^{k-1}f(T^{i}x)$, which
again contradicts (\ref{eqq1}). Therefore $n=2^{p+1}+x_{p+1}$, and
then $m=p$ and $x_{m+1}=l$.

(ii) Let $x\in X$ such that $K_{n}\leq
\sum_{i=1}^{k}f(T^{-i}x)<K_{n+1}$, for  some positive integer $k$.
Let $$p=\max\{N(T^{-i}x); \  1\leq i\leq k\}$$ and remark that
$(T^{-i}x)_{n}=x_{n}$ for $n>p$ and $1\leq i\leq k$. The proof
follows in the same way as in case (i) and we leave the details to
the reader.
\end{proof}


Let $\{F_{t}\}_{t\in\mathbb{R}}$ be  the flow built under the
function $f$ with base automorphism $T$. Notice that
$\{F_{t}\}_{t\in\mathbb{R}}$ is a properly ergodic flow. The
following lemma follows directly from the definition of
$\{F_{t}\}_{t\in\mathbb{R}}$.
\begin{lem}\label{lem}
(i) If $t>0$ then $F_{t}(x,0)\in X \times\{0\}$ if and only if there
exists an integer $k\geq 1$ such that $t=\sum_{i=0}^{k-1}f(T^{i}x)$.

(ii) If $t<0$ then $F_{t}(x,0)\in X \times\{0\}$ if and only if
there exists an integer $k\geq 1$ such that
$t=-\sum_{i=1}^{k}f(T^{-i}x)$.
\end{lem}
\begin{prop}For any $\delta>0$,
\begin{equation}\label{eq0}
\lim_{s\rightarrow\infty}\mu(\Delta_{F,\delta,s}(X))=0.
\end{equation}
\end{prop}
\begin{proof}
By Lemma \ref{lem} we have
\begin{equation}\label{eq}
\mu(\Delta_{F,\delta,s})=\mu\left(\left\{x\in X; \exists
k\in\mathbb{N}, e^{s-\delta}<\sum_{i=0}^{k-1}f(T^{i}x)<
e^{s+\delta}\right\}\bigcup\right.
\end{equation}
\begin{equation*}
\left.\left\{x\in X; \exists k\in\mathbb{N},
e^{s-\delta}<\sum_{i=1}^{k}f(T^{-i}x)< e^{s+\delta}\right\}\right).
\end{equation*}
Proposition \ref{prop} implies that
\begin{equation}\label{eq1}
\mu\left(\left\{x\in X; \exists k\in \mathbb{N}, K_{n}\leq
\sum_{i=0}^{k-1}f(T^{i}x)<K_{n+1}\right\}\right)\leq
\frac{1}{2^{[\log_{2}{n}]+1}},
\end{equation}
\begin{equation}\label{eq2}
\mu\left(\left\{x\in X; \exists k\in \mathbb{N}, K_{n}\leq
\sum_{i=1}^{k}f(T^{-i}x)<K_{n+1}\right\}\right)\leq
\frac{1}{2^{[\log_{2}{n}]+1}}.
\end{equation}

Notice that for $s$ sufficiently large,
$(e^{s-\delta},e^{s+\delta})$ intersects at most two consecutive
intervals $[K_{n},K_{n+1})$. This, together (\ref{eq}), (\ref{eq1})
and (\ref{eq2}) implies (\ref{eq0}).
\end{proof}
From Proposition \ref{PrB} and Proposition 5.3 we can then conclude:
\begin{cor}
The flow $\{F_{t}\}_{t\in\mathbb{R}}$ constructed above satisfies
Property B.
\end{cor}

\end{document}